\documentclass[a4paper]{article}
\usepackage{amsthm,amssymb,amsmath,enumerate,graphicx,epsf,bbold}

\newcommand{\COLORON}{0}
\newcommand{\NOTESON}{1}
\newcommand{\Debug}{0}
\usepackage[usenames,dvips]{color} 
\usepackage{amsthm,amssymb,amsmath,enumerate,graphicx,epsf,stmaryrd}
\usepackage[dvips, bookmarks, colorlinks=false, breaklinks=true]{hyperref}

\newcommand{\comment}[1]{}
\newcommand{\COMMENT}[1]{}

\definecolor{darkgray}{rgb}{0.3,0.3,0.3}
\newcommand{\defi}[1]{{\color{darkgray}\emph{#1}}}

\comment{
	\begin{lemma}\label{}	
\end{lemma}
\begin{proof}

\end{proof}

\begin{theorem}\label{}
\end{theorem} 
\begin{proof} 	

\end{proof}

}

\newtheorem{proposition}{Proposition}[section]

\newtheorem{theorem}[proposition]{Theorem}

\newtheorem{lemma}[proposition]{Lemma}

\newtheorem{conjecture}{{\color{red}Conjecture}}[section]

\newtheorem{problem}[conjecture]{{\color{red}Problem}}

\newtheorem{examp}[proposition]{Example}

\newcommand{\FIG}{0}

\ifnum \NOTESON = 1 \newcommand{\note}[1]{ 

\hspace*{-30pt}
	{\color{blue}  NOTE: \color{Turquoise}{\small  \tt \begin{minipage}[c]{1.1\textwidth}  #1 \end{minipage} \ignorespacesafterend }} 
	
	}
\else \newcommand{\note}[1]{} \fi

\newcommand{\afsubm}[1]{ \ifnum \Debug = 1 {\mymargin{#1}}
\fi} 

\ifnum \Debug = 1 
\else  \fi

\ifnum \FIG = 1 
\else  \fi

\ifnum \FIG = 1 
\else  \fi

\ifnum \Debug = 1 \usepackage[notref,notcite]{showkeys}
\fi

\ifnum \COLORON = 0 \renewcommand{\color}[1]{}
\fi

\newcommand{\cc}{\ensuremath{\mathcal C}}

\newcommand{\sm}{\backslash}

\newcommand{\g}{\ensuremath{G\ }}

\newcommand{\Prb}[1]{Problem~\ref{#1}}

\newcommand{\Cnr}[1]{Con\-jecture~\ref{#1}}

\renewcommand{\iff}{if and only if}
\newcommand{\fe}{for every}
\newcommand{\Fe}{For every}

\newcommand{\st}{such that}

\newcommand{\ti}{there is}

\newcommand{\wrt}{with respect to}

\newcommand{\labtequ}[2]{
 \begin{equation} \label{#1} 	\begin{minipage}[c]{0.9\textwidth}  #2 \end{minipage} \ignorespacesafterend \end{equation} }

\newcommand{\mymargin}[1]{
  \marginpar{%
    \begin{minipage}{\marginparwidth}\small%
      \begin{flushleft}%
        {\color{blue}#1}%
      \end{flushleft}%
   \end{minipage}%
  }%
}%

\newcommand{\mySection}[2]{}

\title{Delay colourings of cubic graphs}
\author{Agelos Georgakopoulos
\medskip 
\\
  {Mathematics Institute}\\
 {University of Warwick}\\
  {CV4 7AL, UK}\\}

\date{}
\begin{document}
\maketitle

\begin{abstract}
In this note we prove the conjecture of  \cite{HaWiWi} that every bipartite multigraph with integer edge delays admits an edge colouring with $d+1$ colours in the special case where $d=3$. A connection to the Brualdi-Ryser-Stein conjecture is discussed.
\end{abstract}

\section{Introduction}

Motivated by scheduling issues in optical networks, Wilfong Haxell and Winkler \cite{HaWiWi} made the following elegant combinatorial conjecture:

\begin{conjecture} \label{conjdelay}
Let \g be a bipartite multigraph with partition classes $A,B$ and maximum degree $d$. Suppose that each edge $e$ is associated with an integer `delay' $r(e)$. Then $G$ admits an edge-colouring $f: E(G)\to \{0,\ldots,d\}$ \st\ $f$ is proper on $A$ and $f+r(mod\ d+1)$ is proper on $B$.
\end{conjecture}

See also \cite{alon_edge_2007}.
When the graph consists of just two vertices joined by $d$ parallel edges, this is implied by a theorem of Hall \cite{hallAbelian} as noted in \cite{HaWiWi}. This fact has been dubbed `the fundamental theorem of juggling' (think of the two vertices as the two hands of a juggler juggling $d$ balls).

More generally, one can consider the case where the edges impose an arbitrary distortion of the colours, given by a permutation $r$ of the set of colours, rather than `delaying' the colour by a constant:

\begin{problem} \label{p}
Let \g be a bipartite multigraph with partition classes $A,B$ and maximum degree $d$. Suppose that each edge $e$ is associated with a permutation $r_e$ of the set $\{0,1,\ldots, d\}$. Then $G$ admits a proper distortion-colouring with colours $\{0,1,\ldots, d\}$ (definitions in the next section). 
\end{problem}

It turns out that in this generality the problem becomes much harder than its special case in \Cnr{conjdelay}: as noticed by N.~Alon (private communication), the special case of \Prb{p} where each class $A,B$ cosists of just one vertex is equivalent to the following conjecture of \cite{AhBeRai}, which is a strengthening of a well-known conjecture of Brualdi and Stein on transversals in Latin squares.

\begin{conjecture}[{\cite[Conjecture 2.4\footnote{The conjecture of \cite{AhBeRai} is more general than \Cnr{a}: it allows for larger sets $V_2,V_3$ with an appropriate modification of the perfect matching condition.}]{AhBeRai}}] \label{a}
Let $H$ be a 3-partite 3-uniform hypergraph, with partition classes $V_1,V_2,V_3$ \st\ $|V_2| =|V_3|=|V_1| + 1$. Suppose that for every $x \in V_1$, the set of hyperedges containing $x$ induces a perfect matching of $V_2 \cup V_3$. Then $V_1$ is matchable.
\end{conjecture}
Here, $V_1$ being matchable means that there is a matching in $H$ containing all vertices in $V_1$.

Indeed, to see the equivalence, represent each edge in \Prb{p} by a vertex in $V_1$, and let $V_2, V_3$ be sets of size $d+1=|V_1+1|$, to be thought of as the colour on the left endvertex and the colour on the right endvertex respectively.

\Cnr{a} strengthens the following well-known conjecture, made independently by  Brualdi \cite{BrualdiRyser} and Stein \cite{SteTra}


\begin{conjecture}[Brualdi-Stein] \label{latin}
In every $n \times n$ Latin square there exists a transversal of size $n-1$.
\end{conjecture}
An $n \times n$ matrix with entries in $\{1,\ldots, n\}$ is a \defi{Latin square}, if no two entries in the same row or in the same column are equal. A \defi{transversal} is a set of entries, each in a different row and in a different column, and each containing a different symbol. (For odd $n$, Ryser \cite{ryser} conjectured that there is even a transversal of size $n$.)

To see that \Cnr{a} implies \Cnr{latin}, construct a 3-partite hypergraph $H$ with $V_1$ being the set of rows of that Latin square $L$, $V_2$ the set of columns, and $V_3= \{1,\ldots, n\}$, and for each entry of $L$ introduce an edge containing the three corresponding vertices of $H$. Then delete an arbitrary vertex in $V_1$ with all edges containing it to obtain the setup of \Cnr{a}.

\medskip

The aim of this note is to prove that \Prb{p} has a positive answer for $d=3$.


\section{Definitions}

Let $G=(V,E)$ be a bipartite multigraph of maximum degree $d$ with bipartition $\{A,B\}$. Let  $Col:= \{0,1,\ldots, d\}$ be the set of \defi{colours}, and suppose that every edge $e\in E$ is associated with  a bijection $r_e$ on $Col$, called the \defi{distortion} of $e$; intuitively, we are going to colour $e$ at its $A$ end and the colour will be distorted by $r_e$ when seen from $B$ (in \cite{HaWiWi} $r_e$ was addition,  mod $d+1$,  with the `delay' of $e$). If $a,b$ is the endvertex of $e$ in $A,B$ respectively, then we use the notation $ba(\cdot)$ to denote $r_e(\cdot)$ and $ab(\cdot)$ to denote $r^{-1}_e(\cdot)$.

A $k$-colouring of a set $E$ (of edges) is a function $f : E \to \{0,1,\ldots, k-1\}$. 
A $k$-colouring $f$ of the edges of a multigraph $G$ as above is called a \defi{proper distortion-colouring} (\wrt\ the permutations $r_e$), if \fe\ $a\in A$ we have $f(e)\neq f(g)$ \fe\ two edges $e,g$ incident with $a$ and \fe\ $b\in B$ we have $r_e(f(e))\neq r_g(f(g))$ \fe\ two edges $e,g$ incident with $b$.

\section{Main}

\begin{theorem}
\Fe\ bipartite multigraph \g of maximum degree 3, and any edge distortions, 
\ti\ a proper distortion-colouring of $E(G)$ with 4 colours 0,1,2,3.
\end{theorem}
\begin{proof}
We may assume without loss of generality that every vertex of \g has degree precisely 3, for otherwise we can add some dummy edges to make \g cubic. It is well known that the edges of a regular bipartite multigraph can be decomposed into disjoint perfect matchings \cite[Corollary 2.1.3]{DiestelBook05}. So let $M,M',M''$ be perfect matchings of \g with $M \cup M' \cup M''=E(G)$. Then $M' \cup M''$ is a 2-factor, and it can be decomposed into a collection \cc\ of edge disjoint cycles.

Let $\{A,B\}$ be the bipartition of $V(G)$. We are going to let each element of \cc\ choose the colours of the edges of $M$ incident with its $A$ side. More precisely, given a $C \in \cc$, let $M_{C\cap A}$ denote the set of edges in $M$ incident with $C\cap A$, and let $M_{C\cap B}$  denote the set of edges in $M \sm M_{C\cap A}$ incident with $B$. We are going to prove that 
\labtequ{MA}{\Fe\ $C \in \cc$, \ti\ a 4-colouring $f_A$ of $M_{C\cap A}$ \st\ \fe\ 4-colouring $f_B$ of $M_{C\cap B}$, there is a 4-colouring $f_C$ of $E(C)$ \st\ $f_A \cup f_B \cup f_C$  is a proper distortion-4-colouring.}
Note that \eqref{MA} easily implies a proper distortion-colouring of $E(G)$ with 4 colours: the sets $M_{C\cap A}\mid C\in\cc$ are pairwise edge disjoint, and their union is $M$. Thus, we can begin by colouring each of them by a colouring $f_A$ as in \eqref{MA}, and then we can extend the colouring to each $C\in\cc$ keeping it proper.

So let us prove \eqref{MA}. Given a $C \in \cc$, pick a 2-edge subarc $uvy$ of $C$ with $u,y\in A$. Distinguish two cases:

If the distortions of the edges $uv,vy$ are identical, then give the edges $m_u,m_y$ of $M$ incident with $u,y$ colours that are different (when seen from $A$). If $C$ happens to be a 2-cycle, in which case $u=y$, give $m_u=m_y$ any colour.

If those distortions are not identical, then colour (the $A$ side of) both $m_u,m_y$  with a colour $\alpha$  \st\ $vu(\alpha)\neq vy(\alpha)$. 

In both cases, colour the rest of $M_{C\cap A}$ arbitrarily; those colours will not matter.

We claim that this colouring $f_A$ has the desired property. To prove this, let $f_B$ be any colouring of $M_{C\cap B}$, and note that \fe\ edge $e\in E(C)$ the set $L_e$ of still available colours for $e$, that is, the colours that would not conflict with $f_A \cup f_B$ if given to $e$ on its $B$ side, say, has at least 2 elements; indeed, only 2 edges adjacent with $e$ have been coloured so far and we had 4 colours to begin with. 

Let us first deal with the case where $C$ is not a 2-cycle, and  consider again the two edges $vu,vy$ as above.
We claim that $L_{vu}\neq L_{vy}$; indeed, the colours we gave to $m_u,m_y$ were chosen is such a way so as to forbid a different candidate colour at the $v$ side of each of $vu,vy$, and so $L_{vu}\neq L_{vy}$ holds. 
On the other hand, the colour of the  edge in $M$ incident with $v$ forbids the same colour for each of $vu,vy$, which implies that $L_{vu}\cap L_{vy}\neq \emptyset$. 

Thus, since $L_{vu}, L_{vy}$ are neither equal nor disjoint, and each contains at least two colours, we can find a common colour $\beta \in L_{vu}\cap L_{vy}$ and another 2 colours $\gamma\in L_{vu}, \delta\in L_{vy}$ so that $\beta, \gamma,\delta$ are all distinct. Now colour $vu$ with $uv(\gamma)$ (so that its colour seems to be $\gamma$ on its $B$ side),
and note that our colouring is still proper, since this colour came from the allowed list. Consider the next edge $ux$ of $C$ incident with $u$. This edge still has at least 1 available colour after we coloured $uv$ (recall that $|L_e|\geq 2)$, so give it that colour. Continue like this along $C$, to properly colour all its edges except  the last edge $vy$. Now note that when we coloured $vu$ we still left 2 colours available for the $B$ side of $vy$, namely $\beta,\delta\neq \gamma$. At least one of them is still available now, and we assign it to the $B$ side of $vy$ completing the proper distortion-colouring of $C$.

If $C$ is a 2-cycle then the situation is much simpler, and it is straightforward to check that \eqref{MA} holds by distinguishing two cases according to whether its 2 edges bear the same distortions.

This completes the proof. Note that we proved something stronger than \eqref{MA}: for each $C \in \cc$, all but one of the edges in $M_{C\cap A}$ can be precoloured arbitrarily. 
\end{proof}

\comment{
\begin{problem}
Prove that \cite[Conjecture 1.1]{HaWiWi} is true \iff\ it is true for arbitrary distortions instead of delays.
\end{problem}
}

\bibliographystyle{plain}
\bibliography{delay}
\end{document}